\documentclass[11pt]{amsart}
\usepackage{graphicx,amscd,color}

\theoremstyle{plain}
\newtheorem{theorem}{Theorem}[section]

\newtheorem{lemma}[theorem]{Lemma}

\theoremstyle{remark}
\newtheorem{remark}[theorem]{Remark}

\title [Reduction of bridge positions along a bridge disk]
{Reduction of bridge positions along a bridge disk}

\author[J. H. Lee]{Jung Hoon Lee}
\address{Department of Mathematics and Institute of Pure and Applied Mathematics,
Chonbuk National University, Jeonju 54896, Korea}
\email{junghoon@jbnu.ac.kr}

\begin{document}

\begin{abstract}
Suppose a knot in a $3$-manifold is in $n$-bridge position.
We consider a reduction of the knot along a bridge disk $D$ and
show that the result is an $(n-1)$-bridge position if and only if
there is a bridge disk $E$ such that $(D, E)$ is a cancelling pair.
We apply this to an unknot $K$, in $n$-bridge position with respect to a bridge sphere $S$ in the $3$-sphere,
to consider the relationship between a bridge disk $D$ and a disk in the $3$-sphere that $K$ bounds.
We show that if a reduction of $K$ along $D$ yields an $(n-1)$-bridge position,
then $K$ bounds a disk that contains $D$ as a subdisk and intersects $S$ in $n$ arcs.
\end{abstract}

\maketitle

\section{Introduction}\label{sec1}

Every closed orientable $3$-manifold $M$ can be decomposed into two handlebodies $V$ and $W$ with common boundary $S$
and it is called a {\em Heegaard splitting} of $M$, denoted by $M = V \cup_S W$.

For a Heegaard splitting $M = V \cup_S W$,
suppose that there are disks $D \subset V$ and $E \subset W$ such that $|D \cap E| = 1$.
Then the Heegaard splitting is said to be {\em stabilized}.
A stabilized Heegaard splitting can be destabilized.
By compressing $V$ along $D$ (or $W$ along $E$), we get a lower genus Heegaard splitting.
Conversely, it is known that if a compression of $V$ along an essential disk $D$
yields a lower genus Heegaard splitting,
then there is a disk $E$ in $W$ such that $|D \cap E| = 1$ \cite{Gordon}.

Let $K$ be a knot in $M$.
The notion of Heegaard splitting can be extended to the pair $(M, K)$.
For $M = V \cup_S W$, suppose that $V \cap K$ and $W \cap K$ are collections of $n$ boundary parallel arcs.
The decomposition $(M, K) = (V, V \cap K) \cup_S (W, W \cap K)$ is called a {\em bridge splitting} of $(M, K)$,
and we say that $K$ is in {\em $n$-bridge position} with respect to $S$.
Each arc of $V \cap K$ and $W \cap K$ is a {\em bridge}.
A bridge $a$ cobounds a {\em bridge disk $D$} with an arc $b$ in $S$,
i.e. $\partial D = a \cup b$ and $a \cap b = \partial a = \partial b$, such that $D \cap K = a$.

Let $a$ be a bridge in, say, $V$.
An isotopy of $a$ to $b$ along a bridge disk $D$ and further, slightly into $W$
will be called a {\em reduction} in this paper.
We consider when a knot is also in bridge position after a reduction.
Not all reductions result in bridge positions. See Figure \ref{fig1} for an example.
A bridge splitting is {\em perturbed} if there are bridge disks $D \subset V$ and $E \subset W$
such that $D$ intersects $E$ at one point of $K$.
We call $(D, E)$ a {\em cancelling pair}.
It is known that for a perturbed bridge splitting,
a reduction along $D$ (or $E$) results in a bridge position \cite[Lemma 3.1]{Scharlemann-Tomova}.
We show that the converse is also true.
This phenomenon is analogous to the case of Heegaard splitting.

\begin{theorem}\label{thm1}
Let $K$ be a knot in $n$-bridge position in a $3$-manifold $M$.
A reduction of $K$ along a bridge disk $D$ yields an $(n-1)$-bridge position
if and only if there is a bridge disk $E$ such that $(D, E)$ is a cancelling pair.
\end{theorem}

\begin{figure}[ht!]
\begin{center}
\includegraphics[width=8cm]{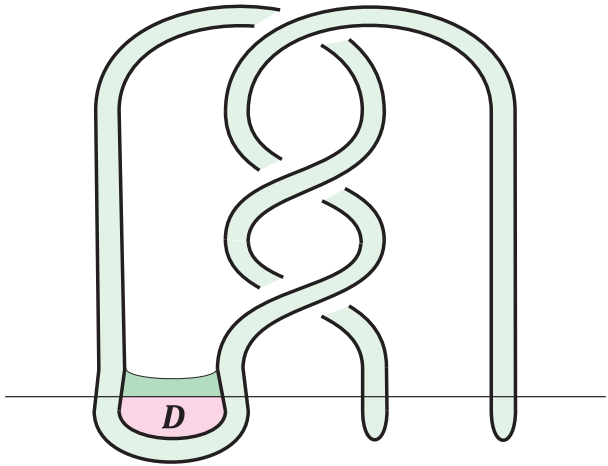}
\caption{A reduction that does not yield a bridge position}\label{fig1}
\end{center}
\end{figure}

We apply Theorem \ref{thm1} to an unknot in $S^3$.
An unknot $K$ in $n$-bridge position can be contained in a $2$-sphere $P$
that intersects the bridge sphere $S$ in a single loop \cite{Otal}, \cite{Ozawa}.
Hence each of the two disks that $K$ bounds in $P$ intersects $S$ in $n$ arcs.
We study when a disk intersecting $S$ in $n$ arcs contains a given bridge disk, and
show that the condition on the bridge disk as in Theorem \ref{thm1} guarantees it.
Such a disk containing a bridge disk might be useful to study a collection of disjoint bridge disks,
hence furthermore to study a collection of disjoint compressing disks.
However, we remark that the converse of Theorem \ref{thm2} does not hold (Figure \ref{fig1}).

\begin{theorem}\label{thm2}
Let $K$ be an unknot in $n$-bridge position in $S^3$.
If a reduction of $K$ along a bridge disk $D$ yields an $(n-1)$-bridge position,
then $K$ bounds a disk $F$ such that
\begin{itemize}
\item $F$ contains $D$ as a subdisk, and
\item $F$ intersects the bridge sphere in $n$ arcs.
\end{itemize}
\end{theorem}

In Section \ref{sec2}, we give a proof of Theorem \ref{thm1}.
In Section \ref{sec3}, we prove Theorem \ref{thm2} using a cancelling pair.

\section{Proof of Theorem \ref{thm1}}\label{sec2}

Let $(M, K) = (V, V \cap K) \cup_S (W, W \cap K)$ be a bridge splitting and let $K$ be in $n$-bridge position.
One implication that if $(D, E)$ is a cancelling pair,
then a reduction of $K$ along $D$ yields an $(n-1)$-bridge position
is well-known, e.g. \cite[Lemma 3.1]{Scharlemann-Tomova}, so we omit the proof here.

Now we go the other direction.
Suppose a reduction of $K$ along a bridge disk $D$ yields an $(n-1)$-bridge position.
We assume that $D$ is in $V$ without loss of generality.
Denote $D \cap K$ by $a$ and $D \cap S$ by $b$.
By a reduction along $D$, $a$ is isotoped past $b$ to an arc $a'$ in $W$.
Consider the rectangle $R$ with opposite sides $a'$ and $b$.
It is a part in $W$ of the region which $a$ swept out. 
See Figure \ref{fig2}.
We denote the other two sides of $R$ by $\alpha$ and $\beta$.

\begin{figure}[ht!]
\begin{center}
\includegraphics[width=12.5cm]{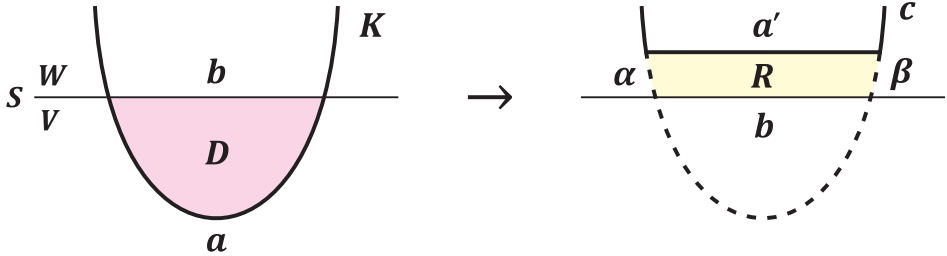}
\caption{A reduction along $D$}\label{fig2}
\end{center}
\end{figure}

The arc $a'$ is a subarc of a new bridge $c$.
We denote a bridge disk for $c$ by $C$ and let $d = C \cap S$.
The bridge disk $C$ and the rectangle $R$ have $a'$ in common and they may have other intersection.
Suppose that there exists an arc component of $R \cap C$ with an endpoint in $\mathrm{int} \, a'$.
This happens if, around the axis $a'$, $C$ spirals with respect to $R$.
Let $\gamma$ be such an arc with a point $x$ of $\partial \gamma$ closest to the point $y = a' \cap \alpha$.
We push a collar neighborhood of the line segment $\overline{xy}$ in $C$
to the opposite side of $R$ as in Figure \ref{fig3}.
Then $\gamma$ is changed into an arc with an endpoint in $\mathrm{int} \, \alpha$.
In this way, even if $C$ spirals more than $360^{\circ}$,
we isotope every arc component of $R \cap C$ with an endpoint in $\mathrm{int} \, a'$
so that the endpoint is not in $\mathrm{int} \, a'$.
Hence we may assume that there exists no arc component of $R \cap C$ with an endpoint in $\mathrm{int} \, a'$.
(In fact, we do not necessarily need to isotope an endpoint
into $\mathrm{int} \, \alpha$ (or $\mathrm{int} \, \beta$).
See Figure \ref{fig4} for an example.)

\begin{figure}[ht!]
\begin{center}
\includegraphics[width=12.5cm]{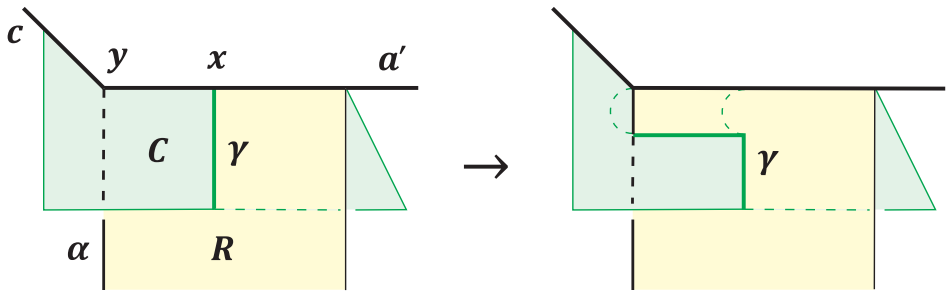}
\caption{An isotopy of $C$}\label{fig3}
\end{center}
\end{figure}

\begin{figure}[ht!]
\begin{center}
\includegraphics[width=12.5cm]{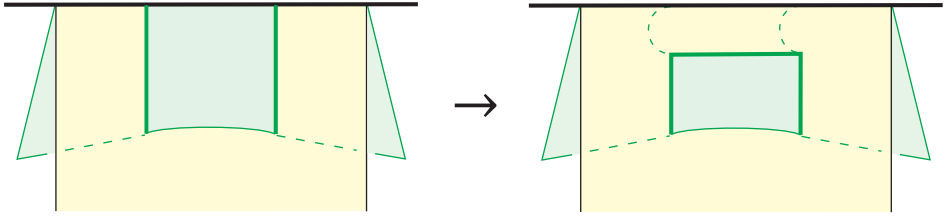}
\caption{Two arcs are merged into one.}\label{fig4}
\end{center}
\end{figure}

A circle component of $R \cap C$ can be removed out of $R$ by an isotopy passing through $\alpha$ or $\beta$.
An arc component of $R \cap C$ belongs to one of the six types below according to its endpoints.

\begin{enumerate}
\item Both endpoints are in $\alpha$.
\item Both endpoints are in $\beta$.
\item Both endpoints are in $b$.
\item One endpoint is in $\alpha$ and the other is in $b$.
\item One endpoint is in $\beta$ and the other is in $b$.
\item One endpoint is in $\alpha$ and the other is in $\beta$.
\end{enumerate}

Every arc of types $(1)$, $(2)$, $(3)$, $(4)$, $(5)$ can be removed out of $R$ by an isotopy (Figure \ref{fig5}),
but an arc of type $(6)$ cannot be removed.
So we may assume that $R \cap C$ consists of type $(6)$ arcs and
the number of components $k = |R \cap C| - 1$ is minimal. 
See Figure \ref{fig6}.
It will be shown that $k = 0$.
So for this purpose we suppose that $k > 0$.

\begin{figure}[ht!]
\begin{center}
\includegraphics[width=7cm]{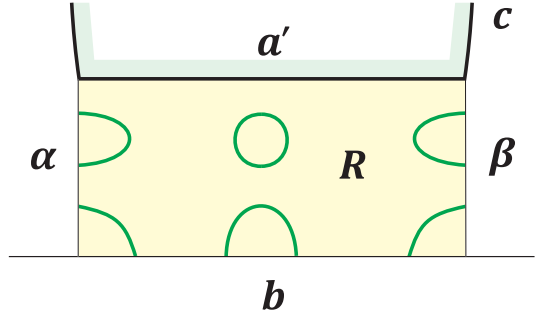}
\caption{Components of $R \cap C$ that can be removed}\label{fig5}
\end{center}
\end{figure}

\begin{figure}[ht!]
\begin{center}
\includegraphics[width=7cm]{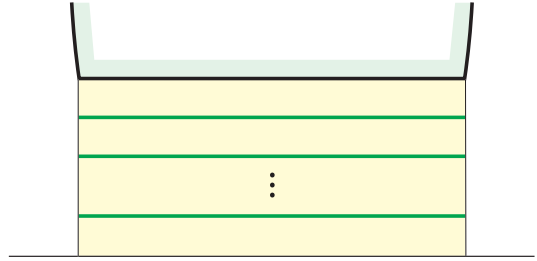}
\caption{The intersection $R \cap C$}\label{fig6}
\end{center}
\end{figure}

Now we reverse the reduction operation.
The arc $a'$ moves back to $\alpha \cup a \cup \beta$ along $R \cup D$ and the bridge disk $C$ moves together.
Of course, some neighborhoods of the $k$ arcs of $R \cap C$ are pushed into $V$,
producing $k$ parallel lower caps for $D$. 
See Figure \ref{fig7}.
After this isotopy, $C$ contains $D$ as a subdisk and $C \cap W$ is a $k$-punctured disk.
Let $C' = C \cap W$.

\begin{figure}[ht!]
\begin{center}
\includegraphics[width=7cm]{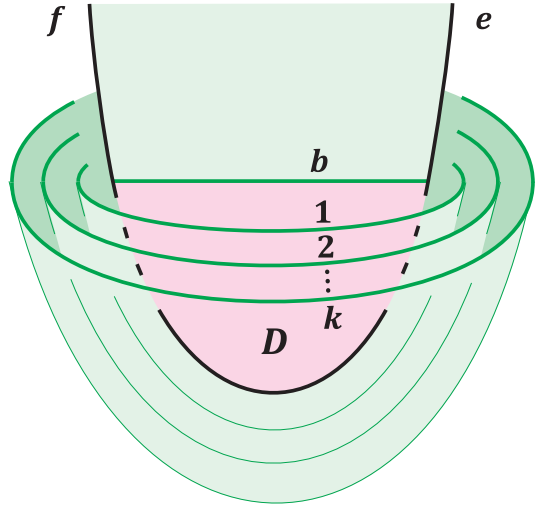}
\caption{The disk $C$ after reversing the reduction operation}\label{fig7}
\end{center}
\end{figure}

Let $e$ and $f$ be the two bridges in $W$
that are adjacent to the bridge $a$ in the original $n$-bridge position of $K$.
We denote a bridge disk for $e$ by $E$.
We choose $E$ such that $|C' \cap E|$ is minimal.

The component of $\partial C'$ which is not a puncture can be regarded as a $4$-gon
consisting of two bridges $e$ and $f$ and two arcs $b$ and $d$ in $S$.
We give label $1$ to the puncture of $C'$ which is closest to $b$ in $S$, and
label consecutively the other nested punctures in $S$ by $2, \ldots, k$.
See Figure \ref{fig7} and Figure \ref{fig8}.

\begin{figure}[ht!]
\begin{center}
\includegraphics[width=7.5cm]{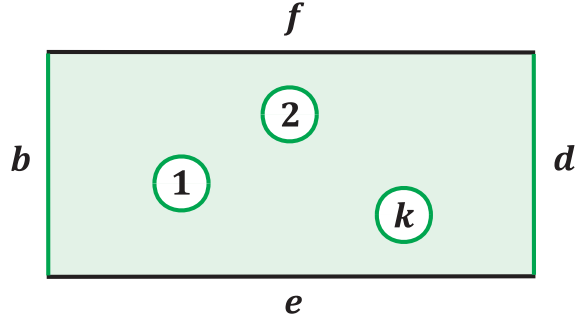}
\caption{The $k$-punctured disk $C'$}\label{fig8}
\end{center}
\end{figure}

If $E$ spirals with respect to $C'$ around the axis $e$,
then as before we isotope an endpoint of an arc
of $C' \cap E$ in $\mathrm{int} \, e$ into $\mathrm{int} \, d$.
See Figure \ref{fig9}.
So we may assume that there is no arc component of $C' \cap E$ with an endpoint in $\mathrm{int} \, e$.

\begin{remark}\label{rmk1}
After the isotopy in Figure \ref{fig9}, the arc $E \cap S$ does not intersect $d$ minimally in $S - K$.
In other words, there is a bigon formed by $(E \cap S) \cup d$ bounding a disk in $S - K$.
But what is needed in our argument is that $(E \cap S) \cup b$ does not make a bigon bounding a disk in $S - K$.
\end{remark}

\begin{figure}[ht!]
\begin{center}
\includegraphics[width=10cm]{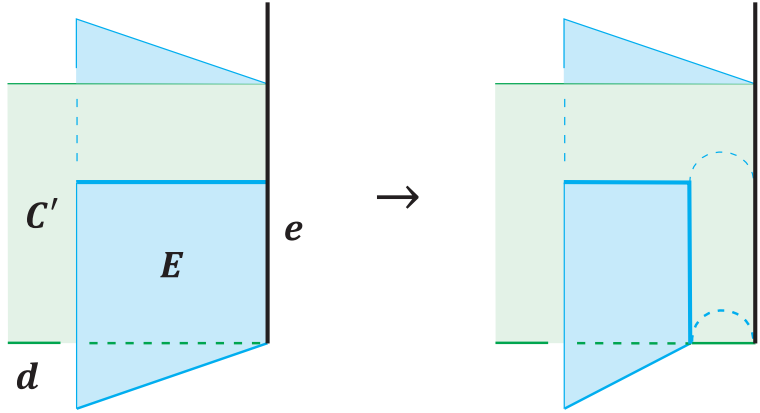}
\caption{An isotopy of $E$}\label{fig9}
\end{center}
\end{figure}

\begin{lemma}\label{lem1}
The intersection $C' \cap E$ consists of arcs, each of which is either
\begin{itemize}
\item the arc $e$, or
\item not boundary parallel in $C'$, or
\item boundary parallel in $C'$ with one endpoint in $b$ and the other in $d$.
\end{itemize}
\end{lemma}

\begin{proof}
Suppose that there is a circle component of $C' \cap E$ which is inessential in $C'$.
(Here, inessential means that it bounds a disk.)
Let $\gamma$ be an innermost such component in $C'$ and let $\Delta$ be the corresponding innermost disk.
Then a disk surgery of $E$ along $\Delta$ reduces $|C' \cap E|$, a contradiction.
Hence all circle components of $C' \cap E$ are essential in $C'$.
Let $\gamma$ be an innermost circle component of $C' \cap E$ in $E$ and
$\Delta$ be the corresponding innermost disk.
By a disk surgery of $C'$ along $\Delta$, we get a disk that replaces $C'$ and has fewer punctures,
which contradicts the minimality of the number of punctures.
So we may assume that $C' \cap E$ consists of arc components.

Suppose that there exists an arc of $C' \cap E$ (except for $e$) such that
the arc is boundary parallel in $C'$ and
it is not the case that one endpoint of the arc is in $b$ and the other endpoint is in $d$.
Then both endpoints are in $b$, or in $d$, or in a puncture.
We choose an outermost such arc $\gamma$ in $C'$ and let $\Delta$ be the corresponding outermost disk.
A surgery of $E$ along $\Delta$ reduces $|C' \cap E|$, a contradiction.
So we conclude that any boundary parallel arc of $C' \cap E$ in $C'$ has one endpoint in $b$ and the other in $d$.
\end{proof}

We give each endpoint of arcs of $C' \cap E$ a label among $\{ b, d, 1, 2, \ldots, k \}$.
Consider a sequence of labels that appears along $E \cap S$.
It begins with $b$ and $1, 2, \ldots, k$ follow by Remark \ref{rmk1} and
because the $k$ punctures in $S$ are nested around $b$.
Then a possibly empty subsequence of $d$'s follows and then $k, \ldots, 1$ follow and then $b$ follows again.
This pattern is repeated and the sequence ends with $d$.
See Figure \ref{fig10} for an appearance of the sequence.
In the figure, $(d)$ means that it can be empty.
We consider a pair of labels of an arc of $C' \cap E$ without order.

\begin{figure}[ht!]
\begin{center}
\includegraphics[width=12.5cm]{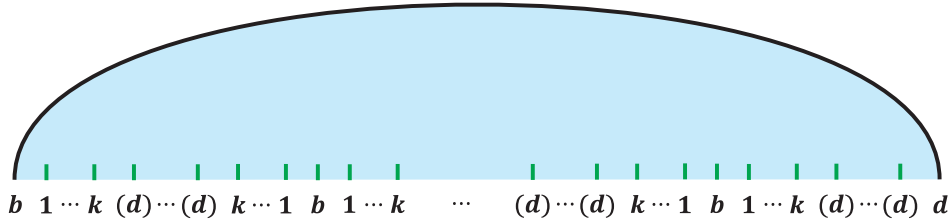}
\caption{A sequence of labels on $E \cap S$}\label{fig10}
\end{center}
\end{figure}

\begin{lemma}\label{lem2}
A pair of labels of an outermost arc of $C' \cap E$ in $E$ is $(k,k)$.
\end{lemma}

\begin{proof}
Let $\gamma$ be an outermost arc of $C' \cap E$ in $E$ and
let $\Delta$ be the corresponding outermost disk.
By the above paragraph, the pair of labels of $\partial\gamma$ is one of the following.

\begin{enumerate}
\item $(b,1)$
\item $(i,i+1)$ $(i=1, \ldots, k-1)$
\item $(k,d)$
\item $(d,d)$
\item $(k,k)$ (There is no $d$ between consecutive two $k$'s.)
\end{enumerate}

\vspace{0.2cm}

Case $(1)$ $(b,1)$\\
By a boundary compression of $C'$ along $\Delta$,
the arc $b$ and the puncture with label $1$ is connected into an arc,
which is still isotopic to $b$ in $S - K$.
In $V$, we can see that $D$ and the lower cap disk is connected by a band, yielding a disk isotopic to $D$.
See Figure \ref{fig11}.
Hence there exists a disk with fewer punctures replacing $C'$, a contradiction.

\begin{figure}[ht!]
\begin{center}
\includegraphics[width=7cm]{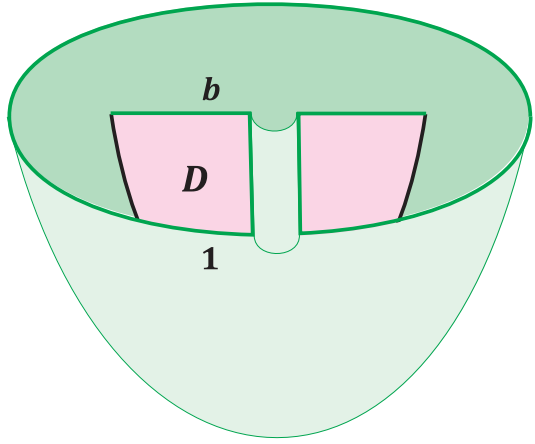}
\caption{A band sum of $D$ and the lower cap disk is isotopic to $D$.}\label{fig11}
\end{center}
\end{figure}

\vspace{0.2cm}

Case $(2)$ $(i,i+1)$ $(i=1, \ldots, k-1)$\\
In this case, a boundary compression of $C'$ along $\Delta$
changes puncture $i$ and $i+1$ into an inessential loop in $S - K$.
So we can reduce the number of punctures by two, a contradiction.

\vspace{0.2cm}

Case $(3)$ $(k,d)$\\
A boundary compression of $C'$ along $\Delta$ connects $d$ and puncture $k$ into an arc.
This also reduces the number of punctures, a contradiction.

\vspace{0.2cm}

Case $(4)$ $(d,d)$\\
Since $\gamma$ is not a boundary parallel arc in $C'$,
a boundary compression of $C'$ along $\Delta$ reduces the number of punctures, a contradiction.

\vspace{0.2cm}

So the pair of labels of $\partial\gamma$ should be the remaining Case $(5)$.
An appearance of the sequence of labels is as in Figure \ref{fig12}.
\end{proof}

\begin{figure}[ht!]
\begin{center}
\includegraphics[width=10.5cm]{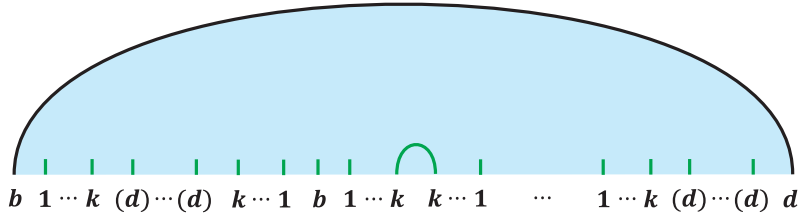}
\caption{A sequence of labels on $E \cap S$}\label{fig12}
\end{center}
\end{figure}

\begin{lemma}\label{lem3}
$k = 0$.
\end{lemma}

\begin{proof}
Consider only all outermost arcs of $C' \cap E$ in $E$.
It is easy to see that there are more than one outermost arcs.
We choose two outermost arcs $\alpha$ and $\beta$ whose endpoints are consecutive along $E \cap S$
such that every other arc with one of its endpoints between $\partial\alpha$ and $\partial\beta$
are parallel to $\alpha$ or $\beta$ as in Figure \ref{fig13}.
By the pattern of the sequence of labels,
all labels $1, \ldots, k-1$ appear between $\partial\alpha$ and $\partial\beta$.
So an arc $\gamma_i$ with a pair of boundary labels $(i,i)$ exists for all $i$ $(i=1, \ldots, k)$.
If we cap off each puncture $i$ of $C'$ with a disk $\Delta_i$, we get a disk $\overline{C}$ from $C'$.
Then $\Delta_i \cup \gamma_i$ can be regarded as a loop with fat vertex in $\overline{C}$.
Consider an innermost one $\Delta_j \cup \gamma_j$ in $\overline{C}$.
Then $\gamma_j$ cuts off a disk from $C'$ and this contradicts Lemma \ref{lem1} that
an arc with label $(i,i)$ is not bounday parallel in $C'$.
The contradiction is caused by the assumption that $k > 0$.
So we conclude that $k = 0$ and $C'$ is a disk.
\end{proof}

\begin{figure}[ht!]
\begin{center}
\includegraphics[width=9cm]{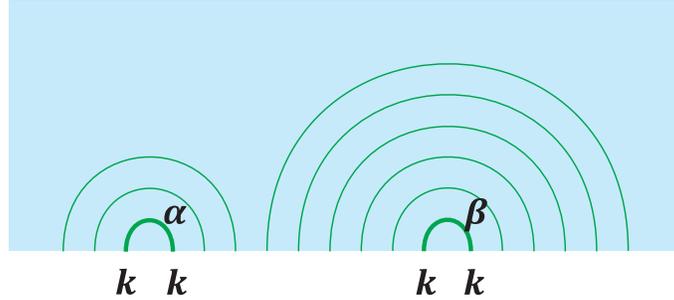}
\caption{Two outermost arcs $\alpha$ and $\beta$ in $E$}\label{fig13}
\end{center}
\end{figure}

By Lemma \ref{lem3}, $C'$ is a disk.
Again we choose $E$ such that $|C' \cap E|$ is minimal.
Suppose that $|C' \cap E| > 0$.
If $E$ spirals with respect to $C'$, 
then we isotope an endpoint of an arc of $C' \cap E$ in $\mathrm{int} \, e$
into $\mathrm{int} \, b$ or $\mathrm{int} \, d$.
By similar argument as in the proof of Lemma \ref{lem1}, $C' \cap E$ (except for $e$) consists of arcs,
each with one endpoint in $b$ and the other in $d$. 
See Figure \ref{fig14}.
Let $\gamma$ be an outermost arc of $C' \cap E$ in $C'$ cutting off an outermost disk $\Delta$ containing $e$.
Let $E_1$ and $E_2$ be the two subdisks of $E$ cut off by $\gamma$ with $e \subset E_1$.
Then $\Delta \cup E_2$ is a new bridge disk for $e$
with $|C' \cap (\Delta \cup E_2)| < |C' \cap E|$, a contradiction.
Therefore $C' \cap E =e$ and this implies that $D \cap E$ is a single point.
This completes the proof of Theorem \ref{thm1}.

\begin{figure}[ht!]
\begin{center}
\includegraphics[width=7.5cm]{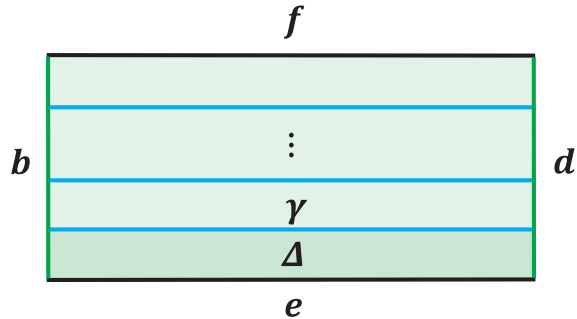}
\caption{The arcs of $C' \cap E$ in $C'$}\label{fig14}
\end{center}
\end{figure}

\section{Proof of Theorem \ref{thm2}}\label{sec3}

By Theorem \ref{thm1}, there exists a bridge disk $E$ such that $(D, E)$ is a cancelling pair.
Let $a = D \cap K$ and $b = E \cap K$ and $p = D \cap E$.
We simultaneously isotope $a$ and $b$ along $D$ and $E$ respectively and
further slightly into $W$ and $V$ respectively, fixing the point $p$.
Let $a'$ be the resulting arc of $a$ after the isotopy and
let $\Delta_1$ be the region that $a$ swept out.
Then $\Delta_1$ contains $D$.
Similarly, let $b'$ be the resulting arc of $b$ after the isotopy and
let $\Delta_2$ be the region that $b$ swept out.

The unknot $K$ is in $(n-1)$-bridge position now and bounds a disk $F_0$ that intersects $S$ in $n-1$ arcs.
The disk $F_0$ and $\Delta_1 \cup \Delta_2$ have $a'$ and $b'$ in common and
they may have other intersection.
If $F_0$ spirals with respect to $\Delta_1$ (or $\Delta_2$) around the axis $a'$ (or $b'$) respectively,
then we isotope $F_0$ as before so that
there is no arc component of $(\Delta_1 \cup \Delta_2) \cap F_0$ with an endpoint in
$\mathrm{int} \, a'$ or $\mathrm{int} \, b'$.
We move all circle components and arc components (except for $a'$ and $b'$)
of $(\Delta_1 \cup \Delta_2) \cap F_0$ horizontally to the outside of $\Delta_1 \cup \Delta_2$
so that the property of $F_0 \cap S$ being $n-1$ arcs is preserved. 
See Figure \ref{fig15}.

\begin{figure}[ht!]
\begin{center}
\includegraphics[width=10cm]{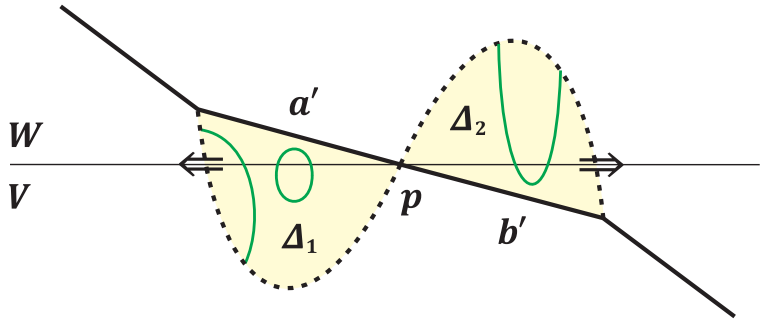}
\caption{Horizontal isotopy of $F_0$}\label{fig15}
\end{center}
\end{figure}

Now we give a perturbation to $K$ using $\Delta_1$ and $\Delta_2$.
The boundary of $F = F_0 \cup \Delta_1 \cup \Delta_2$ is the original $K$ in $n$-bridge position.
But $F \cap S$ is not a collection of $n$ arcs.
We slightly push a part of $F$ in $V$ into $W$ as in Figure \ref{fig16} so that
$F \cap S$ is a collection of $n$ arcs and $F$ still contains $D$ as a subdisk.

\begin{figure}[ht!]
\begin{center}
\includegraphics[width=8cm]{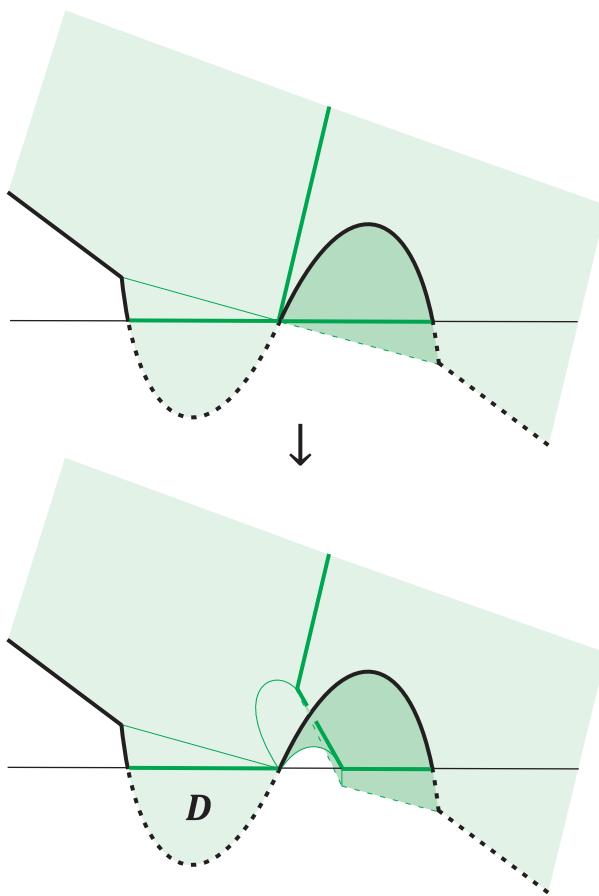}
\caption{An isotopy of $F$}\label{fig16}
\end{center}
\end{figure}

\vspace{0.2cm}

\noindent {\bf Acknowledgements.}
The author would like to thank Toshio Saito for helpful discussion.

\end{document}